\documentclass{article}
\usepackage[a4paper, total={6.3in, 10in}]{geometry}
\usepackage{amsmath,amsthm,amssymb,mathrsfs,amstext, titlesec,enumitem, comment, graphicx, color, xcolor, stmaryrd,mathabx, lipsum,filecontents}
\usepackage{blindtext}
\usepackage{multicol}
\usepackage{wrapfig}
\usepackage[colorlinks=true, linkcolor=blue, citecolor=blue, urlcolor=blue]{hyperref}
\usepackage{xpatch}
\usepackage{xpatch}
\usepackage{tikz-cd}
\makeatletter
\def\Ddots{\mathinner{\mkern1mu\raise\p@
\vbox{\kern7\p@\hbox{.}}\mkern2mu
\raise4\p@\hbox{.}\mkern2mu\raise7\p@\hbox{.}\mkern1mu}}
\makeatother

\titleformat{\section}
{\normalfont\Large\bfseries}{\thesection}{1em}{}
\titleformat*{\subsection}{\Large\bfseries}
\titleformat*{\subsubsection}{\large\bfseries}
\titleformat*{\paragraph}{\large\bfseries}
\titleformat*{\subparagraph}{\large\bfseries}

\makeatother
\theoremstyle{Theorem}
\newtheorem{thm}{Theorem}[section]
\newtheorem{lem}[thm]{Lemma}

\newtheorem{conj}[thm]{Conjecture}

\theoremstyle{definition}
\newtheorem{defn}[thm]{Definition}

\newcommand{\N}{\mathbb{N}}
\newcommand{\Z}{\mathbb{Z}}
\newcommand{\Q}{\mathbb{Q}}

\newcommand{\p}{\mathcal{P}}
\newcommand{\bN}{\beta\mathbb{N}}

\date{\vspace{-5ex}}

\begin{document}

\title{{\bf Moreira's Theorem for Image Partition Regular Matrices}}

\author{ Sayan Goswami \footnote{Ramakrishna Mission Vivekananda Educational and Research Institute, Belur Math, Howrah, West Benagal-711202, India, \textit{ sayan20math@gmail.com}. }}

\date{\vspace{-5ex}}

\makeatother

\maketitle

\begin{abstract}

The famous Hindman conjecture says that for any finite coloring of natural numbers, there exists a monochromatic copy of the form $\{x,x+y,xy\}.$
In a celebrated article, Moreira \cite{anal} gave a partial answer to this conjecture by showing that every finite coloring of the natural numbers contains a monochromatic configuration of the form $\{x, x+y, xy\}$. 
In this article we prove matrix versions (both finite and infinite) of Moreira's theorem. A matrix $A$ is said to be an image partition regular matrix if for any finite coloring of naturals, there exists a monochromatic image of $A,$ i.e. there exists a vector $\vec X$ such that all the entries of $A\vec X$ are monochromatic.
From a recent paper of Bowen, one can derive the finite matrix version of the Moreira theorem: if \(A\) and \(B\) are two finite image partition regular matrices of the same order, then under any finite coloring of \(\mathbb{N}\), there exist vectors \(\vec{X}\) and \(\vec{Y}\) such that all entries in the union of \(A\vec{X}, A\vec{X} + B\vec{Y}, A\vec{X} \cdot B\vec{Y}\) are monochromatic, where $A\vec{X} \cdot B\vec{Y}$ denote the vector each of its entries are pointwise multiplication of the coordinates of $A\vec{X} \text{ and } B\vec{Y}$.

In this article, we give a short combinatorial proof of this result, and then we extend it to infinite image partition regular matrices.

%Our main result proves this theorem for image partition regular matrices: we show that if \(A\) and \(B\) are two finite image partition regular matrices of the same order, then under any finite coloring of \(\mathbb{N}\), there exist vectors \(\vec{X}\) and \(\vec{Y}\) such that all entries in the union of \(A\vec{X}, A\vec{X} + B\vec{Y}, A\vec{X} \cdot B\vec{Y}\) are monochromatic. Furthermore, we extend this result to several natural classes of infinite image partition regular matrices, such as those derived from centrally image partition regular or Milliken–Taylor systems. 
\end{abstract}


\section{Introduction}

The set of positive integers $\N$ exhibits rich combinatorial behavior under finite colorings, as explored in the field of arithmetic Ramsey theory. A fundamental goal in this area is to identify combinatorial structures that are guaranteed to be monochromatic under any finite partition of $\N.$

A collection $\mathcal{F}$ of subsets of $\N$ is said to be \textbf{partition regular} if, for every $F \in \mathcal{F}$ and any finite coloring of $F$, there exists a monochromatic subset $G \subset F$ such that $G \in \mathcal{F}$. And it is called \textbf{weakly partition regular} if for every finite coloring of $\N$, there exists a monochromatic $F \in \mathcal{F}.$

In a similar fashion, an equation $F(x_1, \ldots, x_n) = 0$ over $\N$ is called partition regular if, for every finite partition of $\N$, there exists a monochromatic solution to the equation $F$. The following conjecture is widely known as the Hindman conjecture first appeared in \cite{..}.

\begin{conj}\label{Question}\textup{(Hindman Conjecture \cite{..})}
    The pattern $\{x,y,x+y,xy\}$ is weakly partition regular.
\end{conj}

In \cite{erdos}, Erd\H{o}s first asked if the pattern $\{x,y,x+y,xy\}$ is weakly partition regular.
In $1979$, using brute force computations, Hindman \cite{..} and Graham %\cite{.} 
established an affirmative answer to the question in the case of $2$-colorings. After that, Hindman conjectured that this pattern should be weakly partition regular, which is today known as the Hindman conjecture. Over the finite fields, this was first proved by Green and Sander in \cite{GS}. 
Recently, Bowen \cite{...} provided a combinatorial proof of the fact that the Hindman conjecture is true for $2$-coloring. A major breakthrough came with the work of Moreira \cite{anal}, who proved that for any finite coloring of $\N$, there exist $x, y \in \N$ such that the set $\{x, x+y, x \cdot y\}$ is monochromatic. Subsequently, Alweiss found a concise proof in \cite{al1}. %Over the set of rational numbers, the Hindman conjecture is known to be true due to Bowen and Sabok, and Alweiss independently \cite{AL, BS}. 
Recently in \cite{G25}, the author of this article proved the quotient version of the Hindman conjecture, where it was shown that $\{x,y,x/y,x+y\}$ is weakly partition regular.

For any $k\in \N,$ and sets $A_1,\ldots ,A_k \subseteq \N$, we define
\begin{enumerate}
    \item $A_1 + \cdots +A_k = \{a_1+\cdots +a_k : a_i \in A_i \text{ for }1\leq i\leq k\};$ and  
    \item $A_1  \cdots A_k = \{a_1\cdots a_k : a_i \in A_i \text{ for }1\leq i\leq k\}.$
\end{enumerate}
In \cite[Question 8.4]{kra1}, Kra, Moreira, Richter, and Robertson posed a set-theoretic analogue of Moreira's theorem.  In \cite{b}, Bowen addressed a finitary version of this question by showing that the set $A\cup (A+B) \cup A \cdot B$ is monochromatic where $A$ is infinite, and $B$ is arbitrarily large. The result of \cite{b} was more general and has several strong conclusions, including the finite matrix version of the Moreira theorem, which we discuss now.

In this article, we prove three results, of which the first goal of this article is to give an alternative proof of the finite matrix version of the Moreira theorem, say Theorem \ref{todo}. Then we prove two new results, say Theorem \ref{inf2} and \ref{qq} concerning infinite image partition regular matrices, and then combining these results, we prove an infinite matrix version of the Moreira theorem, say \ref{finalinf}.

%In \cite{b}, Bowen used an induction method of Moreira (see \cite{anal}) to prove his results. But here we will use simple combinatorial and ultrafilter arguments.
But prior to that, let us recall the notions of {\it Image partition regular matrix.}

\subsubsection*{Image Partition Regular Matrices}
Most linear monochromatic patterns can be formulated in terms of the images of \textbf{image partition regular matrices}, making it important to classify such matrices. This theory originated with the works of Rado and Deuber, and was further developed through contributions by many mathematicians.

\begin{defn}
     A matrix $A$ is said to be an \textbf{image partition regular matrix} if, for any finite coloring of $\N$, there exists a vector $\vec{X}$ such that all entries of $A\vec{X}$ are monochromatic. 
\end{defn}

For instance, van der Waerden's theorem \cite{vdw}, which ensures the existence of monochromatic arithmetic progressions, can be formulated in terms of the images of the matrix 
$$
V_l\label{v} = 
\begin{pmatrix} 
1 & 0  \\ 
1 & 1  \\ 
\vdots & \vdots \\ 
1 & l \\
\end{pmatrix},
$$ 
where $l \in \N$.

One of the foundational and, indeed, the first infinitary results in Ramsey theory is Hindman's theorem. This theorem states that, for every finite partition of $\N$, there exists an infinite sequence $\langle x_n \rangle_n$ such that the set consisting of all its finite sums is monochromatic. Before stating this result formally, let us denote by $\mathcal{P}_f(X)$ the set of all nonempty finite subsets of a given nonempty set $X$.

\begin{thm}\textup{\cite{hindman}}
   For any finite coloring of $\N$, there exists a color $A$ and an infinite sequence $\langle x_n \rangle_n$ such that 
   $$
   FS\left(\langle x_n \rangle_n\right) = \left\{ \sum_{n \in \alpha} x_n : \alpha \in \mathcal{P}_f(\N) \right\} \subseteq A.
   $$
\end{thm}

For any sequence $\langle x_n \rangle_n$, sets of the form $FS\left(\langle x_n \rangle_n\right)$ are known as $IP$ sets. Note that any $IP$ set corresponds to the image of a matrix in which each row consists of all zeros except for finitely many $1$s:
$$H=
\begin{pmatrix}
1 & 0 & 0 & \cdots \\
0 & 1 & 0 & \cdots \\
1 & 1 & 0 & \cdots \\
\vdots & \vdots & \vdots & \vdots 
\end{pmatrix}.
$$
Thus, Hindman’s theorem can be reformulated in terms of problems involving infinite image partition regular matrices.  Despite its significance, only a few examples of infinite image partition regular matrices are known. For further examples, refer to \cite[Section 5.4]{hs}. In contrast, finite image partition regular matrices are abundant, and their complete classification is well-documented (see \cite[Chapter 15]{hs}).

The following theorem is our first result. As we mentioned earlier, it can be proved from \cite{b}, but here we give a self-contained proof.

Let us denote by $S\left(\vec{X}\right)$ the set of all entries in a vector $\vec{X}$ over $\N$. 
\begin{thm}\label{todo}
     Let $s\in \N$ and $\langle A_i\rangle_{i=0}^s$ be a sequence of finite image partition regular matrices. Then for any finite coloring of $\N$ there exists a sequence of vectors $\langle \vec x_i\rangle_{i=0}^s$ such that 
\begin{align*}
&\left\lbrace S\left(A_0 \vec x_0\right) \cdots  S\left(A_j \vec x_j\right): 0\leq j\leq s\right\rbrace   \cup  \\   
&\Big\{\left(S\left(A_0 \vec x_0\right) \cdots S\left(A_j \vec x_j\right)\right)+ \left( S\left(A_{j+1} \vec x_{j+1}\right) +\cdots + S\left( A_i \vec x_i\right)\right): 0\leq j<i\leq s\Big\} 
\end{align*}
 is monochromatic.

\end{thm}

Taking $A_1 = A_2 = (1)$, Theorem \ref{todo} implies that $\{x, x+y, xy\}$ is monochromatic. The following examples illustrate two non-trivial applications of Theorem \ref{todo}, where different finite image partition regular matrices are employed. For their partition regularity, see \cite[Chapter 15]{hs}.

\paragraph{Example 1:}
Let 
$$ 
A_1 = 
\begin{pmatrix} 
1 & 0 \\ 
0 & 1 \\ 
1 & 1 
\end{pmatrix}
, 
\quad A_2 = 
\begin{pmatrix} 
1 & 0 \\ 
0 & 1 \\ 
1 & 1 
\end{pmatrix}.
$$ 
Then, there exist elements $\{a, b, c, d\}$ such that the following pattern is monochromatic:
$$
\left\{ 
\begin{array}{ccccccccccc} 
a, & b \\ 
ac, & ad, & bc, & bd, & a+b, & & a+c, & a+d, & b+c, & b+d \\ 
a+b+c, & a+b+d, & a+c+d, & b+c+d \\ 
a+b+c+d, & ac+ad, & bc+bd, & ac+bc, & ad+bd \\ 
ac+bc+ad+bd 
\end{array} 
\right\}.
$$

\paragraph{Example 2:}
Let 
$$ 
A_1 = 
\begin{pmatrix} 
1 & 0 \\ 
0 & 2 \\ 
1 & 2 
\end{pmatrix}
, 
\quad A_2 = 
\begin{pmatrix} 
1 & 0 \\ 
0 & 1 \\ 
1 & 1 
\end{pmatrix}.
$$ 
Then, there exist elements $\{a, b, c, d\}$ such that the following pattern is monochromatic:
$$
\left\{ 
\begin{array}{ccccccccccc} 
a, & 2b \\ 
ac, & ad, & 2bc, & 2bd, & a+2b, & & a+c, & a+d, & 2b+c, & 2b+d \\ 
a+2b+c, & a+2b+d, & a+c+d, & 2b+c+d \\ 
a+2b+c+d, & ac+ad, & 2bc+2bd, & ac+2bc, & ad+2bd \\ 
ac+2bc+ad+2bd 
\end{array} 
\right\}.
$$

Before we proceed to our next result, let us recall some preliminaries about ultrafilters.

\subsubsection*{Ultrafilters}
Let $\beta \N$ denote the collection of ultrafilters, and for any two $p, q \in \beta \N$, we define:
$$
A \in p \cdot q \iff \{x : x^{-1}A \in q\} \in p,
$$
where $x^{-1}A = \{y : xy \in A\}.$

Under this definition, $(\bN, \cdot)$ forms a compact Hausdorff right-topological semigroup, where the right action is continuous. Let $K(\bN, \cdot)$ denote the minimal two-sided ideal of this semigroup. Each idempotent in     $K(\bN,\cdot)$ is a minimal idempotent ultrafilter, and sets belonging to such ultrafilters are referred to as multiplicative Central sets.

Similarly one can define the additive operation on $\bN$ by \[
A \in p + q \iff \{x \in \mathbb{N} : -x + A \in q\} \in p,
\]
where \( -x + A = \{y \in \mathbb{N} : x + y \in A\} \).

Also note that \( (\beta \mathbb{N}, +) \) forms a compact Hausdorff right-topological semigroup and let \( K(\beta \mathbb{N}, +) \) denote the minimal two-sided ideal of this semigroup. Each idempotent in \( K(\beta \mathbb{N}, +) \) is a minimal idempotent ultrafilter, and sets belonging to such ultrafilters are referred to as \emph{additive Central Sets}. A set \( A \subseteq \mathbb{N} \) is called \emph{piecewise syndetic} if there exists \( p \in K(\beta \mathbb{N}, +) \) such that \( A \in p \).

A matrix \( A \) (not necessarily finite) is said to be a \emph{centrally image partition regular matrix} if for any central set \( C \), there exists a vector \( \vec{X} \) such that all the entries of \( A\vec{X} \) are contained in \( C \). In \cite[Theorem 15.24]{hs}, it was proved that every finite image partition regular matrix is centrally image partition regular.

Besides this Hindman matrix \( H \)\footnote{Discussed in the introduction}, there exist many other centrally infinite image partition regular matrices. For several examples, one can see \cite{hls03, hs2000}.

\medskip

In $1975$–$76$, Milliken~\cite{m} and Taylor~\cite{t} independently proved a joint extension of both Hindman's theorem and Ramsey's theorem. However, there are several equivalent formulations of the Milliken–Taylor systems; here we state the version most convenient for our purposes.

Given $F, G \in \mathcal{P}_f(\mathbb{N})$, we write $F < G$ to mean that $\max F < \min G$.

\begin{defn}[Milliken-Taylor system]\label{defmillikentaylor}
    Let $m\in \N,$ let $\langle a_j\rangle_{j=1}^m$ and $\langle x_n\rangle_{n}$ be sequences in $\N.$ The
Milliken-Taylor System determined by $\langle a_j\rangle_{j=1}^m$ and $\langle x_n\rangle_{n}$ is 
$$MT\left(\langle a_j \rangle_{j=1}^m, \langle x_n\rangle_{n}\right)=\left\lbrace \sum_{j=1}^m a_j \cdot \sum_{t\in F_j}x_t:F_1<F_2<\cdots <F_m, \text{ and } F_i\in \p_f(\N) \right\rbrace.$$
\end{defn}

Clearly, the Milliken--Taylor system generated by the sequence $\langle a_j \rangle_{j=1}^m$ is the image of the matrix $MT\left(\langle a_j \rangle_{j=1}^m\right)$, each of whose rows contains finitely many non-zero entries and can be compressed (that is, by deleting zeros and consecutive repeated terms) to the sequence $\langle a_j \rangle_{j=1}^m$.

$\bullet$ It can be proved that a set $A \subseteq \mathbb{N}$ contains the image of the matrix $MT\left(\langle a_j \rangle_{j=1}^m\right)$ if and only if there exists an idempotent ultrafilter $p$ in $(\beta\mathbb{N}, +)$ such that
\[
A \in a_1p + a_2p + \cdots + a_m p
\]
(for details, see \cite[Theorem 1.11]{BHW14}).

\begin{defn}\text{}
  \begin{enumerate}
      \item For any finite matrix $A$ and infinite matrix $B$, let us define 
\[
\begin{pmatrix}
A & B
\end{pmatrix}
\]
to be the matrix each of whose rows is a row of $A$ followed by a row of $B$.

 \item  A matrix $A$ is said to be \textit{piecewise image partition regular} if, for any additive piecewise syndetic set $C$, there exists a vector $\vec{x}$ such that $S(A\vec{x}) \subset C$.

 \end{enumerate} 
\end{defn}

The second goal of this article is to provide two new results concerning image partition regular matrices. The first one is the following.

\begin{thm}\label{inf2}
If $A$ and $B$ are finite (piecewise) and infinite (Centrally) image partition regular matrices, respectively, then the matrix $\begin{pmatrix}
        A & B
    \end{pmatrix}$ is an infinite (piecewise) image partition regular matrix.
\end{thm}

And the second one deals with Milliken-Taylor matrices.

\begin{thm}\label{qq}
Let $\langle a_j \rangle_{j=1}^m$ be a finite sequence of positive integers, and let $A$ be a finite piecewise image partition regular matrix.  
Then every piecewise syndetic set $C \subseteq \mathbb{N}$ contains the image of the matrix
\[
\begin{pmatrix}
    A & MT\left(\langle a_j \rangle_{j=1}^m\right)
\end{pmatrix}.
\]
\end{thm}

And then applying both these theorems, we extend Theorem \ref{todo} for infinite image partition regular matrices.

\begin{thm}\label{finalinf}
     Let $s(>1)\in \N,$ $A_0$ be an infinite matrix of the form Theorem \ref{inf2}, and Theorem  \ref{qq}, and $\langle A_i\rangle_{i=1}^s$ be a sequence of finite image partition regular matrices. Then for any finite coloring of $\N$ there exists a sequence of vectors $\langle \vec x_i\rangle_{i=0}^s$ such that 

     \begin{align*}
&\left\lbrace S\left(A_0 \vec x_0\right) \cdots  S\left(A_j \vec x_j\right): 0\leq j\leq s\right\rbrace   \cup  \\   
&\Big\{\left(S\left(A_0 \vec x_0\right) \cdots S\left(A_j \vec x_j\right)\right)+ \left( S\left(A_{j+1} \vec x_{j+1}\right) +\cdots + S\left( A_i \vec x_i\right)\right): 0\leq j<i\leq s\Big\} 
\end{align*}
 is monochromatic.
\end{thm}

\section{Proof of Our Results}

\subsection{Ramsey Families}

Before proceeding to our original work, we revisit the definition of Ramsey families, which will serve as an essential framework for our original result.

\begin{defn}[\textbf{Ramsey family}]\cite[Definition 1.1]{anal}
    Let $k, s \in \N$, and let $f_1, \dots, f_k: \N^s \to \Z$.
We say that $\{f_1, \dots, f_k\}$ is a \emph{Ramsey family} if, for any finite coloring $\N = C_1 \cup \cdots \cup C_r$, there exists $\mathbf{x} \in \N^s$ and $i \in \{1, \dots, r\}$ such that $\big\{f_1(\mathbf{x}), \dots, f_k(\mathbf{x})\big\} \subset C_i$.
\end{defn}

To prove our Theorem \ref{todo} we need to generalize  \cite[Theorem 1.4]{anal} in the sense of Folkman's theorem. The following theorem was proved in \cite[Theorem 1.4]{anal} for polynomials with coefficients from positive integers. However, for our purposes, we require a formulation involving polynomials with rational coefficients.

\begin{thm}\label{M}
    Let $s \in \N$, and for each $i = 1, \dots, s$, let $F_i$ be a finite set of polynomials over $\N^i$ with rational coefficients such that, for all $f \in F_i$ and any $x_1, \dots, x_{i-1} \in \N$, the function $x \mapsto f(x_1, \dots, x_{i-1}, x)$ is a polynomial with a constant term of $0$.

Then for any finite coloring of $\N$, there exists a color $C \subset \N$ and infinitely many $(s+1)$-tuples $x_0, \dots, x_s \in \N$ such that:
$$
\{x_0 \cdots x_s\} \cup \Big\{x_0 \cdots x_j + f(x_{j+1}, \dots, x_i) : 0 \leq j < i \leq s, f \in F_{i-j} \Big\} \subset C.
$$

In fact one can fix $x_1, \dots, x_s \in \N$ such that the collection of $x_0$ with the above conclusion is a piecewise syndetic set.
\end{thm}

Note that the proof of Theorem \ref{M} is similar to the proof of \cite[Theorem 1.4]{anal}, and so we will return to the proof of Theorem \ref{M} in the Appendix. The following theorem is the reformulation of Theorem \ref{todo} for \textit{Ramsey families.}

\begin{thm}\label{main}
    Let $s, d_0, d_1, \dots, d_s \in \N$, and:

    \begin{enumerate}
        \item For each $i = 1, \dots, s$, let $F_i$ be a finite set of polynomials over $\N^{i}$ with rational coefficients such that, for all $f \in F_i$ and any $x_1, \dots, x_{i-1} \in \N$, the function $x \mapsto f(x_1, \dots, x_{i-1}, x)$ is a polynomial with a constant term of $0$.

        \item For each $i = 0, \dots, s$, let $R_i$ be a finite Ramsey family of linear maps defined from $\N^{d_i}$ to $\Z$.
    \end{enumerate}

Then for any finite coloring of $\N$, there exists a color $C \subset \N$ and infinitely many $(s+1)$-tuples $\mathbf{x_0}, \dots, \mathbf{x_s}$, where each $\mathbf{x_i} \in \N^{d_i}$, such that for each $f_i \in R_i$ $(0 \leq i \leq s)$:
$$
\{f_0(\mathbf{x_0}) \cdots f_s(\mathbf{x_s})\} \cup \Big\{f_0(\mathbf{x_0}) \cdots f_j(\mathbf{x_j}) + f(f_{j+1}(\mathbf{x_{j+1}}), \dots, f_i(\mathbf{x_i})) : 0 \leq j < i \leq s, f \in F_{i-j} \Big\} \subset C.
$$

Additionally, our proof guarantees that 
\begin{align}\label{number}
    &\Big\{ x: \{ x\cdot f_1({\vec x_1\,})\cdots f_s({\vec x_s\,})\} \cup \nonumber &\\
    &\{x\cdot f_1({\vec x_1\,})\cdots 
f_j({\vec x_j\,})+f( f_{j+1}({\vec x_{j+1}}\,),\ldots,
f_i({\vec x_i}\,)):0\leq j<i\leq s, f\in F_{i-j}\}\subseteq C\Big\}&
\end{align}
is a piecewise syndetic set.
\end{thm}

\subsection{Proof of Theorem \ref{main}}\label{2}
We require the following technical lemma to establish Theorem \ref{main}:

\begin{lem}\label{lem}
    Let $s, d_0, d_1, \ldots, d_s \in \N$, and for each $i = 0, \dots, s$, let $R_i$ be a Ramsey family of linear maps defined from $\N^{d_i}$ to $\Z$. Then, for every finite coloring of $\N$, there exists a color $C \subset \N$ and $(s+1)$-tuples $\vec{y_0}, \dots, \vec{y_s}$, where for each $i \in \{0, 1, \ldots, s\}$, $\vec{y_i} \in \N^{d_i}$, such that:
    $$
    \left\lbrace \prod_{i \in T} f_i(\vec{y_i}) : \emptyset \neq T \subseteq \{0, \ldots, s\} \text{ and for each } i \in T, f_i \in R_i \right\rbrace \subseteq C.
    $$
\end{lem}

The following lemma establishes the connection between multiplicative Central sets and finite image partition regular matrices.

\begin{lem}\label{obs}
    Let $d\in\N$ and let $R$ be a Ramsey family of linear maps from $\N^d$ to $\Z$.
Let 
$$M=\{p\in\bN:(\forall C\in p)(\exists \vec x\,\in\N^d)(\{f(\vec x\,):f\in R\}\subseteq C)\}\,.$$
Then $M$ is an ideal of $(\beta\N,\cdot)$. Consequently, if $p\in K(\bN,\cdot)$ and 
$C\in p$, then $\{f(\vec x\,):f\in R\}\subseteq C$.
\end{lem}
\begin{proof}
    Let ${\cal G}=\big\{\{f(\vec x\,):f\in R\}:\vec x\,\in\N^d\big\}$.
The ${\cal G}$ is weakly partition regular by the definition of Ramsey family
so by \cite[Theorem 5.7]{hs}, $M\neq\emptyset$. Now, one can immediately verify that $M$ is a two-sided ideal of $(\bN,\cdot )$ and hence $K(\bN,\cdot )\subseteq M.$ This completes the proof.
\end{proof}

Now we are in the position to prove Lemma \ref{lem}.
\begin{proof}[\textbf{Proof of Lemma \ref{lem}:}]
Let $p$ be a minimal idempotent ultrafilter and $C\in p.$ Then by \cite[Lemma 4.14]{hs} $C^\star=\{n\in C:n^{-1}C\in p\}\in p.$ From Lemma \ref{obs} there exists  ${\vec y_0}\in \N^{d_o}$ such that $\{f_0({\vec y_0}):f_0\in R_0\}\subseteq C^\star.$ By induction assume that there exist ${\vec y_0},\ldots ,{\vec y_{s-1}}$ with ${\vec y_i} \in \N^{d_i}$ such that 
 $$D=\left\lbrace \prod_{i\in T} f_i({\vec y_i}) : \emptyset \neq T\subseteq \{0,\ldots ,s-1\}\text{ and for each } i\in T,\, f_i\in R_i \right\rbrace \subseteq C^\star.$$
 Again from \cite[Lemma 4.14]{hs} we have $E=C^\star \cap \bigcap_{n\in D}n^{-1}C^\star \in p.$ Then again from  Lemma \ref{obs} there exists ${\vec y_s}\in \N^{d_s}$ such that $\{f_s({\vec y_s}):f_s\in R_s\}\subseteq E.$ Now one can easily check that $$\left\lbrace \prod_{i\in T} f_i({\vec y_i}) : \emptyset \neq T\subseteq \{0,\ldots ,s\}\text{ and for each }i\in T,\, f_i\in R_i \right\rbrace \subseteq C.$$
 This completes the induction and the proof.
    
\end{proof}
In our proof, we will use the following finitary version of the above lemma.
If the reader is familiar with the \textit{``color compactness argument''} they may skip the proof of Lemma \ref{new1}.

\begin{lem}\label{new1}
     Let $r, s,d_0,d_1,\ldots ,d_s\in\N$, 
and for each $i=0,\ldots,s$,  let $R_i$ be a Ramsey family of linear maps 
defined from $\N^{d_i}$ to $\Z$. There exists $P\in\N$ such that 
whenever $\chi:[1,P]\to \{1,\ldots ,r\}$, there exist $\vec y_0,\vec y_1,\ldots,
\vec y_s$ with each $\vec y_i \in \N^{d_i}$ such that
$\chi$ is constant on $$\left\lbrace \prod_{i\in T} f_i({\vec y_i}) :\emptyset \neq T\subseteq 
\{0,\ldots ,s\}\hbox{ and for each }i\in T\,, f_i\in R_i \right\rbrace.$$
\end{lem}

\begin{proof}
    Let $r,m\in \N,$ be given, and suppose the conclusion fails. For each $n\in \N,$ choose an $r$-coloring $\phi_n$ of $\{1,2,\ldots ,n\}$ such that for any $\vec y_0,\vec y_1,\ldots,
\vec y_s$ with each $\vec y_i \in \N^{d_i}$ the pattern $$\left\lbrace \prod_{i\in T} f_i({\vec y_i}) :\emptyset \neq T\subseteq 
\{0,\ldots ,s\}\hbox{ and for each }i\in T\,, f_i\in R_i \right\rbrace$$ is not monochromatic. Choose by the pigeon hole principle an infinite subset $B_1\subseteq \N$ and an element $\sigma (1)\in \{1,\ldots ,r\}$ such that for all $k\in B_1,$ $\phi_k(1)=\sigma (1).$ Inductively, given $l\in \N,$ with $l>1$ and an infinite subset $B_{l-1}$ of $\N$ choose an infinite subset $B_l$ of $B_{l-1}$ and an element $\sigma (l)\in \{1,\ldots ,r\}$ such that $\min B_l\geq l$ for all $k\in B_l,$ $\phi_k(l)=\sigma (l).$

Define an $r$-coloring $\tau$ of $\N$ by $\tau (l)=\sigma (l).$ Then $\N=\bigcup_{i=1}^r\tau^{-1}[\{i\}]$ so fix by \cite[Corollary 5.9]{hs} some $i\in \{1,\ldots ,r\}$ and vectors $\vec y_0,\vec y_1,\ldots,
\vec y_s$ with each $\vec y_i \in \N^{d_i}$ such that $$D=\left\lbrace \prod_{i\in T} f_i({\vec y_i}) :\emptyset \neq T\subseteq 
\{0,\ldots ,s\}\hbox{ and for each }i\in T\,, f_i\in R_i \right\rbrace \subseteq \tau^{-1}[\{i\}].$$ Pick $k\in \N$ such that $D \subseteq \{1,\ldots ,k\}$ and pick $n\in B_k.$ We claim that $D \subset \{1,\ldots ,k\}$ is monochromatic with respect to $\phi_n,$ a contradiction. To see this, let $a\in D \subseteq \{1,\ldots ,k\}$ be given. Note that $B_k\subseteq B_{k-1}\subseteq \cdots \subseteq B_1$ and so choose any $n\in B_k.$ Then by definition of $\sigma$ we have $$\phi_n(a)=\phi_n(k)\implies \sigma (a)=\sigma (k).$$
Hence the coloring is constant on $\{1,\ldots ,k\}$ and hence on $D.$
This completes the proof.
\end{proof}

Now we are ready to prove Theorem \ref{main}. 

\begin{proof}[\textbf{Proof of Theorem \ref{main}:}]  We may presume that $\overline 0:\N^i \to \{0\}$ is a member of $F_i$ for each $i\in\{1,2,\ldots ,s\}$.
Let $r\in\N$ and let $\omega:\N\to \{1,\ldots ,s\}$ be a coloring of $\N$.
Let $P$ be as guaranteed by Lemma \ref{new1} for $r, s,d_0,d_1,\ldots ,d_s$ 
and $R_0,R_1,\ldots,R_s$.

\begin{center}
    Define $\omega':\N\to \bigtimes_{i=1}^P\{1,\ldots ,r\}$ by,
for $\alpha\in\N$ and $i\in \{1,\ldots ,P\}$, $\omega'(\alpha)_i=\omega(i\alpha)$.
\end{center}
So, given $\alpha,\beta\in\N$,

\begin{center}
$\omega'(\alpha)=\omega'(\beta)\Leftrightarrow 
(\forall i\in \{1,\ldots ,P\})(\omega(i\alpha)=\omega(i\beta)$.
\end{center}

For $0\leq j<i\leq s$, $f\in F_{i-j}$, and $\{a_0,\ldots,a_i\}\subseteq[1,P]$, let 
$$f_{a_{j+1},\ldots,a_i}^{a_0,\ldots,a_j}(x_{j+1},\ldots,x_i)={1\over a_0\cdots a_j}f(a_{j+1}x_{j+1},\ldots,a_ix_i)\,.$$
For $l\in \{1,\ldots ,s\}$, let
$$G_l=\Big\{f_{a_{j+1},\ldots,a_i}^{a_0,\ldots,a_j}:0\leq j<i\leq s\,,\,l=i-j\hbox{, and }\{a_0,\ldots,a_i\}\subseteq[1,P]\Big\}.$$
Then each $G_l$ is a finite set of polynomials over $\N^l$ as required by Theorem \ref{M} so one may
pick $x_0,\ldots,x_s$ in $\N$ such that $\omega'$ is constant on 
$$D=\{x_0\cdots x_s\}\cup\big\{x_0\cdots x_j+g(x_{j+1},\ldots,x_i):0\leq j<i\leq s, g\in G_{i-j}\big\}\,.$$
Since we have that $\overline 0\in G_l$ for each $l\in \{1,\ldots ,s\}$,  
we have that $x_0\cdots x_j\in D$ for each $j\in \{0,1,\ldots ,s-1\}$.

Let $v\in \bigtimes_{i=1}^P\{1,\ldots ,r\}$ be the constant value of $\omega'$ on $D$.
\begin{center}
    Define $\chi:[1,P]\to \{1,\ldots ,r\}$ by $\chi(i)=v_i$.
\end{center}
By Lemma \ref{new1}, pick
$\vec y_0,\vec y_1,\ldots,\vec y_s$ with each $\vec y_i \in \N^{d_i}$
such that $\chi$ is constant on 
$$C=\left\lbrace \prod_{i\in T} f_i({\vec y_i\,}) :\emptyset \neq T\subseteq 
\{0,\ldots ,s\}\hbox{ and for each }i\in T\,, f_i\in R_i \right\rbrace\,.$$
Let $u$ be the constant value of $\chi$ on $C$.
We claim that $\omega$ is constantly equal to $u$ on the set $C\cdot D$. To see this, let $a\in C$ and
$b\in D$. Then $$u=\chi(a)=v_a=\omega'(b)_a=\omega(ab).$$
\begin{itemize}
    \item For $i\in\{0,1,\ldots ,s\}$, define $\vec x_i=x_i\vec y_i$.  
\end{itemize}
To complete the proof we show that whenever $f_i\in R_i$ for each $i\in\{0,1,\ldots ,s\}$,
$$\left\lbrace f_0({\vec x_0\,})\cdots f_s({\vec x_s\,})\}\cup\Big\{f_0({\vec x_0\,})\cdots 
f_j({\vec x_j\,})+f( f_{j+1}({\vec x_{j+1}}\,),\ldots,
f_i({\vec x_i}\,)):0\leq j<i\leq s, f\in F_{i-j}  \right\rbrace \subseteq C \cdot D\,.$$
So assume that $f_i\in R_i$ for each $i\in \{0,1,\ldots ,s\}$.
Then $$\prod_{i=0}^s f_i(\vec x_i\,)=x_0\cdots x_s\cdot\prod_{i=0}^s f_i(\vec y_i\,)\in D \cdot C.$$
Now let $0\leq j<i\leq s$ and let $f\in F_{i-j}$. 
\begin{itemize}
    \item 
For $l\in \{0,1,\ldots ,i\}$, let
$a_l=f_l(\vec y_l\,)$ and let $g=f_{a_{j+1},\ldots,a_i}^{a_0,\ldots,a_j}$. \end{itemize} Then
$g\in G_{i-j}$ and

\begin{align*}
    &f_0({\vec x_0\,})\cdots 
f_j({\vec x_j\,})+f\left( f_{j+1}({\vec x_{j+1}}\,),\ldots,
f_i({\vec x_i}\,)\right)&\\
&=f_0(x_0{\vec y_0\,})\cdots 
f_j(x_j{\vec y_j\,})+f\left( f_{j+1}(x_{j+1}{\vec y_{j+1}}\,),\ldots,
f_i(x_i{\vec y_i}\,)\right) &\\
&=x_0\cdots x_j\cdot \prod_{l=0}^j f({\vec y_l})+F(x_{j+1}a_{j+1},\ldots ,x_ia_i)&\\
&=\big(x_0\cdots x_j+g(x_{j+1},\ldots,x_i)\big)\cdot
\prod_{l=0}^j f_l(\vec y_l\,)\in D \cdot C\,.&
\end{align*}
This completes the proof.
    
\end{proof}

\subsection{Proof of Theorems \ref{finalinf}}

Before we proceed, let us recall the following definition.
\begin{defn}[\textbf{Infinite Ramsey Family}]
Let \( \langle f_i\rangle_{i=1}^\infty: \N^\N \to \Z \) be an infinite sequence of linear maps.
We say that \( \langle f_i\rangle_{i=1}^\infty \) is:
\begin{enumerate}
    \item A \emph{Piecewise Ramsey family} if, for any piecewise syndetic set \( A \subseteq \N \), there exists a vector \( \mathbf{x} \in \N^\N \) such that \( \{f_i(\mathbf{x}) : i \in \N\} \subseteq A \).
    
    \item A \emph{Central Ramsey family} if, for any central set \( A \subseteq \N \), there exists a vector \( \mathbf{x} \in \N^\N \) such that \( \{f_i(\mathbf{x}) : i \in \N\} \subseteq A \).
\end{enumerate}
\end{defn}

The following lemma immediately follows from Theorem \ref{main}.

\begin{lem}\label{main1}
Assume the hypothesis of Theorem \ref{main} except $d_0=\omega,$ the first infinite cardinal number, and the family $R_0=\langle f_{0,i}\rangle_{i=1}^\infty$  is a Piecewise  Ramsey family.

Then for any finite coloring of $\N$, there exists a color $C \subset \N$ and infinitely many $(s+1)$-tuples $\vec{x_0}, \dots, \vec{x_s}$, where each $\vec{x_i} \in \N^{d_i}$, such that for each $f_i \in R_i$ $(0 \leq i \leq s)$:
$$\left\lbrace f_0({\vec x_0\,})\cdots f_s({\vec x_s\,})\}\cup\Big\{f_0({\vec x_0\,})\cdots 
f_j({\vec x_j\,})+f( f_{j+1}({\vec x_{j+1}}\,),\ldots,
f_i({\vec x_i}\,)):0\leq j<i\leq s, f\in F_{i-j}  \right\rbrace \subseteq C \,.$$
    
\end{lem}

Now we will first prove Theorem \ref{inf2}.

\begin{proof}[\textbf{Proof of Theorem \ref{inf2}}]
    Let $R$ and $S$ be two Ramsey families of linear maps generating the matrices $A$ and $B$, respectively. Since $B$ is a centrally image partition regular matrix, $S$ is a central Ramsey family. Let $C$ be any additive piecewise syndetic set. From \cite[Theorem 1.51]{hs}, we know that $K(\bN,+)$ is the union of minimal left ideals. Hence, there exits a minimal left ideal $L$ such that $p\in L$ and $C\in p.$ Choose a minimal idempotent $e\in L$. Since $e\in L,$ we have $L+e\subseteq L.$ But $L$ is minimal, and $L+e$ is a left ideal, so $L+e=L.$ Hence $$p\in L=L+e\implies p=r+e$$ for some $r\in L.$ As $e$ is an idempotent, we have $p=p + e.$ And so  the set
\[
T = \{n \in \mathbb{N} : -n + C \in e\}\in p.
\]
Hence, $T$ is a piecewise syndetic set.
 Then there exists a vector $\vec{x}$ such that $\{f(\vec{x}) : f \in R\} \subset T.$ Hence, the set
\[
D = \bigcap_{f \in R} (-f(\vec{x}) + C) \in e,
\]
and, so $D$ is a central set. Again, from the construction of $S$, there exists a vector $\vec{y}$ such that $\{g(\vec{y}) : g \in S\} \subseteq D.$

Combining these, we see that all the entries of 
\[
\begin{pmatrix}
A & B
\end{pmatrix}
\begin{pmatrix}
\vec{x} \\
\vec{y}
\end{pmatrix}
\]
are contained in $C.$

\end{proof}

Now we prove Theorem \ref{qq}.

\begin{proof}[\textbf{Proof of Theorem \ref{qq}}]
    Since $C$ is a piecewise syndetic set, there exists a minimal left ideal $L$ of $(\beta\mathbb{N},+)$ and an ultrafilter $p \in L$ such that $C \in p$. Let $q \in L$ be any minimal idempotent. For any idempotent ultrafilter $q$ and $n\in \N,$ the ultrafilter is defined as 
    Define the ultrafilter $A\in n^{-1}q\iff n\cdot A \in q.$ It can be shown that $n^{-1}q$ is an idempotent ultrafilter. Define,
\begin{align*}
r&= a_1\cdot  a_m^{-1}q + a_2 \cdot a_m^{-1}q + \cdots +  a_m\cdot  a_m^{-1}q&\\
    &=a_1\cdot  a_m^{-1}q  + \cdots +a_{m-1} \cdot a_m^{-1}q+  q.&
\end{align*}

Since $q \in L$ and $L$ is a left ideal, it follows that $r \in L$. Moreover, $L + r$ is also a left ideal. But as $r \in L$, we have $L + r = L$ (adapting the similar strategy in the proof of Theorem \ref{inf2}).

    Hence there exists \( s \in L \) such that \( p = s + r \). As \( C \in p \), we have 
\[
D = \{ n \in \mathbb{N} : -n + C \in r \} \in s.
\]
Since \( A \) is a finite piecewise image partition regular matrix, there exists a vector \( \vec{X} \) such that \( S(A \vec{X}) \subset D \). Hence the set
\[
E = \bigcap_{n \in S(A \vec{X})} (-n + C) \in r.
\]
Then, by \cite[Theorem 1.11]{BHW14}, there exists a vector \( \vec{Y} \) such that 
\[
MT\left( \langle a_j \rangle_{j=1}^m, \vec{Y} \right) \subseteq E.
\]
Hence all the entries of 
\[
\begin{pmatrix}
    A & MT\left(\langle a_j \rangle_{j=1}^m\right)
\end{pmatrix}
\begin{pmatrix}
    \vec{X} \\
    \vec{Y}
\end{pmatrix}
\]
are contained in \( C \). This completes the proof.

\end{proof}

 \begin{proof}[Proof of Theorem \ref{finalinf}] Consider the Ramsey family $R_0$ generated by the matrices in Theorem \ref{inf2}, and \ref{qq}. Then $R_0$ is an infinite (piecewise) image partition regular matrix. Now from the Lemma \ref{main1}, the proof of this theorem follows immediately.
  %  Theorem \ref{finalinf} follows from Lemma \ref{main1} and Theorems \ref{inf2} and \ref{qq}.
  \end{proof}

\section*{Acknowledgement}
I am thankful to Neil Hindman and Joel Moreira for invaluable discussions on the previous draft of this paper. The author of this paper is supported by NBHM postdoctoral fellowship with reference no: 0204/27/(27)/2023/R \& D-II/11927.

\section{Appendix}

This section is devoted to the proof of Theorem \ref{M}.  

\subsection{Proof of Theorem \ref{M}}
This section is devoted to the proof of Theorem \ref{M}.  The following theorem is called the polynomial van der Waerden theorem, first proved in \cite{bl}.

\begin{thm}\textup{\cite[Theorem 3.6]{polynomial}}\label{ap1}
     Let $F\subset\Q[x]$ be a finite set of polynomials such that $P(0)=0$ for all $P\in F$. Then for any piecewise syndetic set $A\subseteq \N$ there exists $n\in \N$ such that $A\cap \bigcap_{P\in F}(-P(n)+A)$ is a piecewise syndetic set.
\end{thm}

Now we prove  Theorem \ref{M}. The proof is similar to the proof of \cite{anal}, but for completeness, we give the proof here.

\begin{proof}[Proof of Theorem \ref{M}:] 
In our proof we will assume that $\vec 0: \N^i\to \{0\}$ is a member of $F_i$ for each $i\in \{1,2,\ldots ,s\}.$
Let $\N=\bigcup_{i=1}^rC_i\cup \bigcup_{i=1}^qE_i$ be any finite coloring where each $C_i's$ are piecewise syndetic set and $E_j's$ are not a piecewise syndetic set. 
 We will construct, inductively, four sequences:

   \begin{itemize}
     \item $(t_n)_{n\geq0}$ in $\{1,\dots,r\}$,
     \item $(y_n)_{n\geq1}$ in $\N$ increasing,
     \item $(B_n)_{n\geq0}$ of  piecewise syndetic subsets of $\N$, such that $B_n\subseteq C_{t_n},$ and
     \item $(D_n)_{n\geq1}$ of  piecewise syndetic subsets of $\N$, such that $D_n\subseteq B_{n-1}.$
   \end{itemize}

Following \cite{anal}, We will use the notation  $y(m,n)\in\N$ to denote the product $y(m,n):=y_{m+1}y_{m+2}\cdots y_n$ for any $0\leq m\leq n$, with the convention that the (empty) product $y(n,n)$ equals $1$. Letting $t_0=1$ and $B_0=C_1.$ Applying Theorem \ref{ap1} we have $y_1\in \N$ such that

$$D_1:=B_0\cap\bigcap_{f\in F_1} (B_0-f(y_1))$$ is piecewise syndetic. Now from \cite[Lemma 1.5]{update} $y_1D_1$ is a piecewise syndetic set. 
So pick $p\in K(\bN,+)$ such that $y_1D_1\in p.$ Then $\bigcup_{i=1}^q E_i\notin p$ and so $\bigcup_{i=1}^rC_i\in p.$
Hence there exists $t_1\in \{1,\ldots ,r\}$ such that $B_1=y_1D_1\cap C_{t_1}$ is a piecewise syndetic set.

In general, for $n\geq2$, assume that $(t_m)_{m=0}^{n-1}$, $(y_m)_{m=1}^{n-1}$, $(B_m)_{m=0}^{n-1}$ and $(D_m)_{m=1}^{n-1}$ have been constructed. For each $i\in\{1,\dots,s\}$ and each $f\in F_i$, we define the collection $G_n(f)\subset \Q[x]$ of all polynomials  $g$ of the form
  $$g:z\mapsto y(m_1,n-1)f\big(y(m_1,m_2),\ y(m_2,m_3),\ \dots,\ y(m_i,n-1)\cdot z\big)$$
  for any choice $0\leq m_1<m_2<\cdots<m_i<n$.
  If $s>n$ then we set $G_n(f)$ to be empty. Now from Theorem \ref{ap1} we can find $y_n\in\N$ satisfying

  \begin{equation}\label{eq_proof_main41}
    D_n:=B_{n-1}\cap\bigcap_{i=1}^{\min\{n,s\}}\bigcap_{f\in F_i}\bigcap_{g\in G_n(f)} (B_{n-1}-g(y_n))
  \end{equation}

 is a piecewise syndetic set. Similar as above we have $t_n\in \{1,\ldots ,r\}$ such that $B_n=y_nD_n\cap C_{t_n}$ is a piecewise syndetic set. This finishes the inductive constructions. Now from our construction $B_n\subseteq C_{t_n}$ for every $n\geq 0.$  Moreover, $B_n\subset y_n D_n\subset y_n B_{n-1}$. Iteratively we have

 \begin{equation}\label{eq_proof_maindynamicaltheoremnovo}
  \forall m\leq n,\qquad B_n\subset y(m,n) B_m.
  \end{equation}

Since $(t_n)_{n\geq 0}\subseteq \{1,\ldots ,r\}$, there exists $t\in \{1,\dots ,r\}$ and infinitely many tuples of natural numbers $n_0<\cdots<n_s$ such that $t_{n_i}=t$. Now fix one such tuple $n_0<\cdots<n_s$ with $s<n_0.$

  \smallskip
 \begin{center}
    For each   $k=1,...,s,$ define $x_k=y(n_{k-1},n_k).$
 \end{center}

  First, we claim that 
  \begin{equation}\label{eq_proof_main_dynamical_last}
    \forall\ 0\leq j<i\leq s,\qquad \forall\ f\in F_{i-j},\qquad B_{n_s}\subset x_{j+1}\cdots x_s \cdot \left(B_{n_j}-f(x_{j+1},\dots,x_i)\right).
  \end{equation}

  Now fix $0\leq j<i\leq s$ and $f\in F_{i-j}$.
  Observe that there exists some $g\in G_{n_i}(f)$ such that $f(x_{j+1},\dots,x_i)=g(y_{n_i})/y(n_j,n_i-1)$. Using \eqref{eq_proof_maindynamicaltheoremnovo}, we conclude

  \begin{eqnarray*}
    B_{n_s}&\subset&y(n_i,n_s)B_{n_i}\subset y(n_i,n_s)y_{n_i} \cdot D_{n_i}\\ \text{using \eqref{eq_proof_main41}}&\subset& y(n_i-1,n_s) \cdot \Big(B_{n_i-1}-g(y_{n_i})\Big)\\ \text{using \eqref{eq_proof_maindynamicaltheoremnovo}} &\subset& y(n_i-1,n_s) \cdot \Big(y(n_j,n_i-1)B_{n_j}-g(y_{n_i})\Big)\\  &=&y(n_i-1,n_s) y(n_j,n_i-1) \cdot \Big(B_{n_j}-g(y_{n_i})/y(n_j,n_i-1)\Big)\\ &=& x_{j+1}\cdots x_s\cdot \Big(B_{n_j}-f(x_{j+1},\dots,x_i)\Big).
  \end{eqnarray*}
This proves our claim.

Since $B_{n_j}\subset C_t$ for every $j\in\{0,\dots,s\}$,  from equation \eqref{eq_proof_main_dynamical_last} we have \begin{equation}\label{eq_proof_main1}  B_{n_s} \subseteq C_t\cap  \bigcap_{0\leq j<i\leq s}\bigcap_{f\in F_{i-j}}x_{j+1}\cdots x_s \cdot \Big(C_t-f(x_{j+1},\dots,x_i)\Big)
\end{equation}

and hence this set is piecewise syndetic. Choose $x$ from the set in equation \eqref{eq_proof_main1}. Hence $x\in x_1\cdots x_sC_t$ (letting $j=0$, $i=s$ and $f\equiv0$). Therefore $x_0:=x/(x_1\cdots x_s)\in C_t.$ Finally, for $0\leq j<i\leq s$ and $f\in F_{i-j}$, we have $x\in x_{j+1}\cdots x_s\big(C_t-f(x_{j+1},\dots,x_i)\big)$, so $$x_0\cdots x_j+f(x_{j+1},\dots,x_i)=x/(x_{j+1}\cdots x_s)+f(x_{j+1},\dots,x_i)\in C_t.$$

This finishes the proof of Theorem \ref{M}.
\end{proof}

\nocite{*}
%\printbibliography[nottype=...]

\end{document}